\documentclass{amsart}
\usepackage{amsmath}
\usepackage{amsrefs}
\usepackage[all]{xy}

\date{\today}

\newtheorem{thm}{Theorem}

\newtheorem{lemma}[thm]{Lemma}
\newtheorem{prop}[thm]{Proposition}

\theoremstyle{definition}

\theoremstyle{remark}
\newtheorem{remark}[thm]{Remark}

\def\mathcs{C^{*}}
\newcommand{\cs}{\ensuremath{\mathcs}}

\DeclareMathSymbol{\rtimes}{\mathbin}{AMSb}{"6F}
\newcommand{\ib}{im\-prim\-i\-tiv\-ity bi\-mod\-u\-le}

\newcommand{\sme}{\,\mathord{\mathop{\text{--}}\nolimits_{\relax}}\,}
\def\ibind#1{\mathop{#1\mathord{\mathop{\text{--}}}}\!\Ind\nolimits}

\DeclareMathOperator{\Ind}{Ind}

\DeclareMathOperator{\Prim}{Prim}

\DeclareMathOperator{\id}{id}
\def\set#1{\{\,#1\,\}}
\let\tensor=\otimes
\def\restr#1{|_{{#1}}}

\newbox\hidebox
\def\spechide#1{\setbox\hidebox=\hbox{$#1$}
\hbox to\wd\hidebox{$\box\hidebox^\wedge$\hss}}
\makeatletter
\def\labelenumi{\textnormal{(\@alph\c@enumi)}}
\def\theenumi{\@alph \c@enumi}
\def\labelenumii{\textnormal{(\@roman\c@enumii)}}
\def\theenumii{\@roman \c@enumii}
\newcount\charno
\def\alphapart#1{\charno=96
\advance\charno by#1\char\charno}

\makeatother
\def\<{\langle}
\def\>{\rangle}
\let\ipscriptstyle=\scriptscriptstyle
\def\lipsqueeze{{\mskip -3.0mu}}
\def\ripsqueeze{{\mskip -3.0mu}}
\def\ipcomma{\nobreak\mathrel{,}\nobreak}
\newbox\ipstrutbox
\setbox\ipstrutbox=\hbox{\vrule height8.5pt
width 0pt}
\def\ipstrut{\copy\ipstrutbox}
\def\lip#1<#2,#3>{\mathopen{\relax_{\ipstrut\ipscriptstyle{
#1}}\lipsqueeze
\langle} #2\ipcomma #3 \rangle}
\def\blip#1<#2,#3>{\mathopen{\relax_{\ipstrut
\ipscriptstyle{ #1}}\lipsqueeze\bigl\langle} #2\ipcomma #3 \bigr\rangle}
\def\rip#1<#2,#3>{\langle #2\ipcomma #3
\rangle_{\ripsqueeze\ipstrut\ipscriptstyle{#1}}}
\def\brip#1<#2,#3>{\bigl\langle #2\ipcomma #3
\bigr\rangle_{\ripsqueeze\ipstrut\ipscriptstyle{#1}}}
\def\angsqueeze{\mskip -6mu}
\def\smangsqueeze{\mskip -3.7mu}
\def\trip#1<#2,#3>{\langle\smangsqueeze\langle #2\ipcomma #3
\rangle\smangsqueeze\rangle_{\ripsqueeze\ipstrut\ipscriptstyle{#1}}}
\def\btrip#1<#2,#3>{\bigl\langle\angsqueeze\bigl\langle #2\ipcomma
#3
\bigr\rangle
\angsqueeze\bigr\rangle_{\ripsqueeze\ipstrut\ipscriptstyle{#1}}}
\def\tlip#1<#2,#3>{\mathopen{\relax_{\ipstrut\ipscriptstyle{
#1}}\lipsqueeze \langle\smangsqueeze\langle} #2\ipcomma #3
\rangle\smangsqueeze\rangle}
\def\btlip#1<#2,#3>{\mathopen{\relax_{\ipstrut\ipscriptstyle{
#1}}\lipsqueeze
\bigl\langle\angsqueeze\bigl\langle} #2\ipcomma #3
\bigr\rangle\angsqueeze\bigr\rangle}

\def\ip(#1|#2){(#1\mid #2)}
\def\bip(#1|#2){\bigl(#1 \mid #2\bigr)}
\def\Bip(#1|#2){\Bigl( #1 \bigm| #2 \Bigr)}
\IfFileExists{mathrsfs.sty}{\usepackage{mathrsfs}}
{\let\mathscr\mathcal}
\newcommand{\bundlefont}[1]{\mathscr{#1}}
\newcommand{\A}{\bundlefont A}
\newcommand{\B}{\bundlefont B}
\newcommand\D{\bundlefont D}

\newcommand{\go}{G^{(0)}}
\def\sa_#1(#2;#3){\Gamma_{#1}(#2;#3)}
\newcommand\prima{\Prim A}
\newcommand\clsp{\overline{\operatorname{span}}}
\newcommand\tp{\overline{P}}
\newcommand\OO{\mathcal{O}}

\newcommand\usc{up\-per semi\-con\-tin\-uous}

\newcommand\BI{\B_{I}}
\newcommand\bi{B_{I}}
\newcommand\BqI{\B^{I}}
\newcommand\bqi{B^{I}}
\newcommand\I{\mathcal{I}}
\newcommand\gcgb{\sa_{c}(G;\B)}

\newcommand\gcgbi{\sa_{c}(G;\BI)}
\newcommand\gcgbqi{\sa_{c}(G;\BqI)}
\renewcommand\H{\mathcal{H}}
\newcommand\HH{\mathscr{H}}
\DeclareMathOperator{\Ex}{Ex}
\newcommand\tA{\tilde A}
\newcommand\ia{i_{A}}
\newcommand\oI{\mathscr{I}}
\let\phi\varphi

\expandafter\ifx\csname BibSpec\endcsname\relax\else
\BibSpec{collection.article}{%
    +{}  {\PrintAuthors}                {author}
    +{,} { \textit}                     {title}
    +{.} { }                            {part}
    +{:} { \textit}                     {subtitle}
    +{,} { \PrintContributions}         {contribution}
    +{,} { \PrintConference}            {conference}
    +{}  {\PrintBook}                   {book}
    +{,} { }                            {booktitle}
    +{,} { }                            {series}
    +{,} { \voltext}                    {volume}
    +{,} { }                            {publisher}
    +{,} { }                            {organization}
    +{,} { }                            {address}
    +{,} { \PrintDateB}                 {date}
    +{,} { pp.~}                        {pages}
    +{,} { }                            {status}
    +{,} { \PrintDOI}                   {doi}
    +{,} { available at \eprint}        {eprint}
    +{}  { \parenthesize}               {language}
    +{}  { \PrintTranslation}           {translation}
    +{;} { \PrintReprint}               {reprint}
    +{.} { }                            {note}
    +{.} {}                             {transition}
}
\BibSpec{article}{%
    +{}  {\PrintAuthors}                {author}
    +{,} { \textit}                     {title}
    +{.} { }                            {part}
    +{:} { \textit}                     {subtitle}
    +{,} { \PrintContributions}         {contribution}
    +{.} { \PrintPartials}              {partial}
    +{,} { }                            {journal}
    +{}  { \textbf}                     {volume}
    +{}  { \PrintDatePV}                {date}
    +{,} { \eprintpages}                {pages}
    +{,} { }                            {status}
    +{,} { \PrintDOI}                   {doi}
    +{,} { available at \eprint}        {eprint}
    +{}  { \parenthesize}               {language}
    +{}  { \PrintTranslation}           {translation}
    +{;} { \PrintReprint}               {reprint}
    +{.} { }                            {note}
    +{.} {}                             {transition}
}
\BibSpec{book}{%
    +{}  {\PrintPrimary}                {transition}
    +{,} { \textit}                     {title}
    +{.} { }                            {part}
    +{:} { \textit}                     {subtitle}
    +{,} { \PrintEdition}               {edition}
    +{}  { \PrintEditorsB}              {editor}
    +{,} { \PrintTranslatorsC}          {translator}
    +{,} { \PrintContributions}         {contribution}
    +{,} { }                            {series}
    +{,} { \voltext}                    {volume}
    +{,} { }                            {publisher}
    +{,} { }                            {organization}
    +{,} { }                            {address}
    +{,} { pp.~}                        {pages}
    +{,} { \PrintDateB}                 {date}
    +{,} { }                            {status}
    +{}  { \parenthesize}               {language}
    +{}  { \PrintTranslation}           {translation}
    +{;} { \PrintReprint}               {reprint}
    +{.} { }                            {note}
    +{.} {}                             {transition}
}
\fi

\allowdisplaybreaks

\begin{document}
\title{\boldmath Remarks on the Ideal Structure of Fell Bundle \cs-Algebras}

\author{Marius Ionescu}

\address{Department of Mathematics \\ University of Connecticut,
  Storrs, CT 06269-3009}

\email{ionescu@math.uconn.edu}

\author{Dana Williams}
\address{Department of Mathematics \\ Dartmouth College \\ Hanover, NH
03755-3551}

\email{dana.p.williams@Dartmouth.edu}

\thanks{This research was supported by the Edward Shapiro fund at
  Dartmouth College.}

\begin{abstract}
  We show that if $p:\B\to G$ is a Fell bundle over a locally compact
  groupoid $G$ and that $A=\sa_{0}(\go;\B)$ is the \cs-algebra sitting
  over $\go$, then there is a continuous $G$-action on $\prima$ that
  reduces to the usual action when $\B$ comes from a dynamical system.
  As an application, we show that if $I$ is a $G$-invariant ideal in
  $A$, then there is a short exact sequence of \cs-algebras
  \begin{equation*}
    \xymatrix{0\ar[r]&\cs(G,\BI)\ar[r]
    &\cs(G,\B)\ar[r]&\cs(G,\BqI)\ar[r]&0,} 
  \end{equation*}
where $\cs(G,\B)$ is the Fell bundle \cs-algebra and $\BI$ and $\BqI$
are naturally defined Fell bundles corresponding to $I$ and $A/I$,
respectively.  Of course this exact sequence reduces to the usual one
for \cs-dynamical systems. 
\end{abstract}

\maketitle

\section*{Introduction}
\label{sec:introduction}

An important source of examples of \cs-algebras are constructs
associated to dynamics of some sort.  Coming first on any such list
would be group \cs-algebras followed closely by crossed product
\cs-algebras coming from locally compact automorphism groups.  More
generally, we have the \cs-algebras associated to locally compact
groupoids and groupoid dynamical systems.  A very important extension
of group dynamical systems is given by Fell bundles over groups.  Fell
bundles over groups were originally called (saturated) \cs-algebraic
bundles by Fell, and are studied systematically in
\cite{fd:representations2}*{Chap.~VIII}.  The \cs-algebra of a Fell
bundle over a group $G$ should be considered as a very general type of
crossed product of the \cs-algebra sitting over the unit of $G$.
Following Yamagami \citelist{\cite{yam:pm90}\cite{yam:xx87}}, Kumjian
formalized the notion of a Fell bundle $p:\B\to G$ over a locally
compact groupoid in \cite{kum:pams98}.  In this case, the associated
\cs-algebra is meant to be a general type of crossed product of the
\cs-algebra $A=\sa_{0}(G;\B)$ of $\B$ sitting over the unit space
$\go$.  This is illustrated by the examples in
\cite{muhwil:dm08}*{\S2}.

The \cs-algebras of Fell bundles have been the object of considerable
study starting with the group context
\citelist{\cite{exe:ma02}\cite{beaexe:ms01}
  \cite{exe:pjm00}\cite{exe:jfram97}\cite{exemar:pams97}}, then over
\'etale groupoids \citelist{\cite{ful:hjm98}\cite{fulmuh:ijm98}
  \cite{dkr:ms08}} and eventually in the general setting
\citelist{\cite{muhwil:dm08}
  \cite{kmqw:xx09}\cite{busexe:xx09}\cite{bmz} \cite{bus:xx09}}.  This
note is meant as a first step in a systematic investigation of the
ideal structure of \cs-algebras associated to Fell bundles over
locally compact (Hausdorff) groupoids.  Our first result is to show
that if $p:\B\to G$ is a Fell bundle and $A=\sa_{0}(G;\B)$ is the
\cs-algebra of $\B$ over $\go$, then even though there is no explicit
action of $G$ on $A$, there is a natural $G$-action of $G$ on the
primitive ideal space $\prima$ of $A$ which generalizes the usual
notion when $\B$ is the Fell bundle associated to either a classical
\cs-dynamical system or a groupoid dynamical system.

Once we have a $G$-action on $\prima$, then it makes sense to speak of
an invariant ideal $I$ of $A$.  Our next contribution is that to each
$G$-invariant ideal $I$ in $A$, there are naturally associated Fell
bundles $p_{I}:\BI\to G$ and $p^{I}:\BqI\to G$ corresponding to $I$
and $A/I$, respectively, and a short exact sequence
\begin{equation*}
  \xymatrix{0\ar[r]&\cs(G,\BI)\ar[r]
    &\cs(G,\B)\ar[r]&\cs(G,\BqI)\ar[r]&0} 
\end{equation*}
of \cs-algebras.  This result generalizes the fundamental result in
\cs-dynamical systems that asserts that if $(A,G,\alpha)$ is a
\cs-dynamical system with $I$ an $\alpha$-invariant ideal in $A$, then
there is a short exact sequence
\begin{equation*}
  \xymatrix{0\ar[r]&I\rtimes_{\alpha}G
    \ar[r]&A\rtimes_{\alpha}G\ar[r]&(A/I)\rtimes_{\alpha^{I}}G\ar[r]&0.} 
\end{equation*}
(For example, see \cite{wil:crossed}*{Proposition~3.19}.)  However,
the proof is considerable more subtle in our setting and requires the
Disintegration Theorem for Fell bundles
\cite{wil:crossed}*{Theorem~4.13}.

In subsequent work, we intend to use these elementary results together
with \cite{ionwil:xx09b} to study the Mackey machine for Fell bundles
and the fine ideal structure of the corresponding \cs-algebras.

We adopt the usual conventions in the subject.  Representations of
\cs-algebras will be assumed to be nondegenerate, and homomorphisms
between \cs-algebras will always be $*$-preserving.  We will write
$M(A)$ for the multiplier algebra of a \cs-algebra $A$.  As in
\citelist{\cite{lan:hilbert}\cite{rw:morita}}, we view $M(A)$ as the
set of adjointable operators $\mathcal{L}(A)$ where $A$ is viewed as a
Hilbert module over itself.  We write $\tA$ for the \cs-subalgebra of
$M(A)$ generated by $A$ and $1_{A}$.  Thus $\tA$ is simply $A$ if $A$
has a unit and $A$ with a unit adjoined otherwise.

\section{Preliminaries}
\label{sec:preliminaries}

A Fell bundle $p:\B\to G$ over a locally compact Hausdorff groupoid
$G$ is an \usc\ Banach bundle equipped with a continuous, bilinear,
associative multiplication map $(a,b)\mapsto ab$ from $\B^{(2)}:=
\set{(a,b)\in \B\times\B:(p(a),p(b))\in G^{(2)}}$ to $\B$ and an
involution $b\mapsto b^{*}$ from $\B$ to $\B$ satisfying axioms
(a)--(e) of \cite{muhwil:dm08}*{Definition~1.1}.  We will adopt the
notations and conventions of \cite{muhwil:dm08} and refer to
\cite{muhwil:dm08}*{\S1} for the construction of the associated
\cs-algebra $\cs(G,\B)$ built from the $*$-algebra of continuous
compactly supported sections $\sa_{c}(G;\B)$ of $\B$.  In particular,
our Fell bundles are \emph{saturated} in that
$B(x)B(y):=\clsp\set{ab:\text{$a\in B(x)$ and $b\in B(y)$}}=B(xy)$.
Since we make considerable use of the Disintegration Theorem for Fell
bundles (see \cite{muhwil:dm08}*{Theorem~4.13}), we will need to
assume that all our Fell bundles are \emph{separable} in that $G$ is
second countable and that the Banach space $\sa_{0}(G;\B)$ is
separable.  It is important to keep in mind that the property that
each fibre $B(x)$ is an \ib\ has important consequences.  For example,
it follows that $b^{*}b\ge0$ in the \cs-algebra sitting over $s(b):=
s(p(b))$, and that $\|b^{*}b\|=\|b\|^{2}$ for all $b\in \B$.
Moreover, we have the following useful observation.


\begin{lemma}
  \label{lem-norm-ineq}
  If $p:\B\to G$ is a Fell bundle then $\Vert ab\Vert\le\Vert
  a\Vert\Vert b\Vert$ for all $(a,b)\in\B^{(2)}$.
\end{lemma}
\begin{proof}
  The lemma follows from \cite{muhwil:dm08}*{Lemma~1.2} once we prove
  the following observation (which is probably known to specialists
  but we have been unable to find a reference in the literature): If
  $A,B,C$ are $C^{*}$-algebras, and $_{A}X_{B}$ and $_{B}Y_{C}$ are
  \ib s,\footnote{All that is required here is for $X$ and $Y$ to be a
    right Hilbert modules with actions of $A$ and $B$, respectively,
    coming from homomorphisms into the adjointable operators.  Such
    objects are either called \emph{\cs-correspondences} or
    \emph{right Hilbert bimodules} in the literature.}  then\[ \Vert
  x\otimes y\Vert\le\Vert x\Vert\Vert y\Vert\] for all $x\otimes y\in
  X\otimes_{B}Y$. The proof of this observation follows from the
  computation:
  \begin{align*}
    \Vert x\otimes y\Vert^{2} & =\Vert\rip C <x\otimes y,x\otimes
    y>\Vert=\bigl\Vert\brip C<{\rip B <x,x>}\cdot y,y> \bigr\Vert\\
    & =\bigl\Vert \brip C<{\rip B< x,x>^{1/2}}\cdot y,{\rip B<
      x,x>^{1/2}}\cdot y > \bigr\Vert.
  \end{align*}
  If $T$ is an adjointable operator on $Y$ we know that $\rip C<
  Ty,Ty>\le\Vert T\Vert^{2}\Vert\rip C< y,y>\Vert$.  Since $B$ acts on
  the left on $Y$ via adjointable operators it follows that
  \begin{equation*}
    \Vert
    x\otimes y\Vert^{2}\le\bigl\Vert\Vert\rip B<
    x,x>^{1/2}\Vert^{2}\rip C<
    y,y>\bigr\Vert =\Vert x\Vert^{2}\Vert y\Vert^{2}.\qedhere
  \end{equation*}
\end{proof}

We will also require a number of basics concerning \usc\ Banach
bundles.  The definition of an \usc-Banach bundle, as well as a
collection of basic results and original sources, are given in
\cite{muhwil:dm08}*{Appendix~A}.  Actually, much of what is needed is
covered in detail in \cite{wil:crossed}*{Appendix~C.2}; unfortunately,
all the results in \cite{wil:crossed}*{Appendix~C.2} are stated in
terms of \usc\ \cs-bundles even though the additional \cs-structure is
not always necessary for the result.  For example, the proof of
\cite{wil:crossed}*{Proposition~C.20}, which characterizes convergence
in the total space, makes no use of the \cs-axioms and remains valid
for \usc-Banach bundles by simply replacing ``\usc\ \cs-bundle over
$X$'' by ``\usc-Banach bundle over $X$''.  Therefore we will
sheepishly, but firmly, cite results like
\cite{wil:crossed}*{Proposition~C.20} below for \usc-Banach bundles
with the implicit understanding appropriate Banach bundle result is
valid with essentially the same proof as given in
\cite{wil:crossed}*{Appendix~C.2}. For example, if $p:\B\to X$ is an
\usc-Banach bundle over $X$, then we note that $\sa_{0}(X;\B)$ is a
Banach space in the sup-norm by (the same proof as given in)
\cite{wil:crossed}*{Proposition~C.23}.

Although the total space of a Banach bundle is an important
theoretical tool in groupoid constructs, in the wild Banach bundles
are built by first specifying what the sections should be.  As far as
we know, the next result is originally due to Hofmann
\citelist{\cite{hof:74}\cite{hof:lnm77}}, although the only the
details for continuous bundles were ever published (see, for example,
\cite{fd:representations1}*{Theorem~II.13.18}).\footnote{Recently,
  Buss and Exel have published an version of this result which does
  not even require the base space $X$ to be Hausdorff:
  \cite{busexe:xx09}*{Proposition~2.4}.  In their result, they make
  the extra assumption that $\set{f(x):f\in \Gamma}$ is all of $B(x)$.
  This is required to ensure that the topology is unique when $X$
  fails to be Hausdorff.  In our applications, this extra assumption
  on the fibres is satisfied, so we could also appeal to
  \cite{busexe:xx09}*{Proposition~2.4}.} What is needed is the
following.

\begin{thm}[Hofmann-Fell]
  \label{thm-build-bundle}
  Let $X$ be a locally compact space and suppose that for each $x\in
  X$, we are given a Banach space $B(x)$.  Let $\B$ be the disjoint
  union $\coprod_{x\in X}B(x)$, and let $p:\B\to X$ be the obvious
  bundle.  Suppose that $\Gamma$ is a subspace of sections such that
  \begin{enumerate}
  \item for each $f\in \Gamma$, $x\mapsto \|f(x)\|$ is upper
    semicontinuous, and
  \item for each $x\in X$, $\set{f(x):f\in\Gamma}$ is dense in $B(x)$.
  \end{enumerate}
  Then there is a unique topology on $\B$ such that $p:\B\to X$ is an
  \usc-Banach bundle over $X$ with $\Gamma\subset \sa_{}(X;\B)$.
\end{thm}
\begin{proof}
  This result is stated in \cite{dg:banach}*{Proposition~1.3} with a
  reference to \cite{hof:74}.  With the proviso discussed above, it
  follows from \cite{wil:crossed}*{Theorem~C.25}.
\end{proof}

\begin{remark}
  \label{rem-cts-bdle}
  If $p:\B\to G$ is a Fell bundle over a \emph{group} $G$, then the
  underlying Banach bundle is necessarily continuous.  This was
  observed in \cite{bmz}*{Lemma~3.30} and is apparently due to Exel.
  The idea is that $a\mapsto (a^{*},a)\mapsto a^{*}a\mapsto
  \|a^{*}a\|^{\frac12}$ must be continuous as maps $\B\to \B\times\B
  \to A \to [0,\infty)$.  The same argument shows that if $p:\B\to G$
  is a Fell bundle over a locally compact Hausdorff groupoid $G$, then
  the underlying Banach bundle is continuous if and only if the
  associated \cs-algebra $A:=\sa_{0}(\go;\B)$ is a continuous field
  over $\go$.
\end{remark}

If $p:\B\to G$ is a Fell bundle over a locally compact Hausdorff
groupoid $G$, then
\begin{equation*}
  A:=\sa_{0}(\go;\B)
\end{equation*}
is a \cs-algebra which we call \emph{the \cs-algebra of $\B$ sitting
  over $\go$}.  Note that $A$ is a $C_{0}(\go)$-algebra, and let
$\sigma:\prima\to \go$ be the associated structure map.\footnote{For a
  summary of basic results and our notations for $C_{0}(X)$-algebras,
  please refer to \cite{wil:crossed}*{\S C.1}.}  If $u\in \go$, let
$q_{u}:A\to A(u)$ be the quotient map with kernel $I_{u}$.  Then
$\sigma(P)=u$ if and only if $I_{u}\subset P$.\footnote{For $u\in\go$,
  the fibres $A(u)$ and $B(u)$ are identical as sets.  If there is an
  excuse for using different letters, it is that $A(u)$ is meant to be
  thought of as a \cs-algebra and $B(u)$ is $A(u)$ viewed as a
  $A(u)\sme A(u)$-\ib.}
 
Furthermore, $\prima$ is naturally identified with the disjoint union
of the $\Prim A(u)$ \cite{wil:crossed}*{Proposition~C.5}.  Thus we
will write
\begin{equation*}
  \prima=\set{(u,P):\text{$u\in \go$ and $P\in \Prim A(u)$}}.
\end{equation*}
It will be important to keep in mind that $(u,P)=q_{u}^{-1}(P)$.

We will need the following technical lemma on $C_{0}(X)$-algebras in
the proof of Proposition~\ref{prop-bqi}.

\begin{lemma}
  \label{lem-cox}
  Suppose that $A$ is a $C_{0}(X)$-algebra with structure map
  $\sigma:\prima\to X$ (see \cite{wil:crossed}*{Proposition~C.5}), and
  let $A(x)=A/I_{x}$ be the fibre over $x$.  Let $K$ be an ideal in
  $A$, and let
  \begin{equation*}
    F=\set{P\in\prima:P\supset K}
  \end{equation*}
  be the closed subset of $\prima$ identified with $\Prim (A/K)$.
  Then $\sigma\restr F$ induces a $C_{0}(X)$-structure on $A/K$ and $
  (A/K)(x)= (A/K)/\oI_{x}$, where $\oI_{x}=(I_{x}+K)/K\cong
  I_{x}/(I_{x}\cap K)$.  In particular, $(A/K)(x)\cong A/(K+I_{x})$.
\end{lemma}
\begin{proof}
  Recall that the Dauns-Hofmann Theorem
  \cite{wil:crossed}*{Theorem~A.24} gives an isomorphism $\Phi_{A}$ of
  $C^{b}(\prima)$ onto the center $ZM(A)$ of the multiplier algebra
  which is characterized by
  \begin{equation*}
    (\Phi_{A}(f)a)(P) =f(P)a(P),
  \end{equation*}
  where $a(P)$ denotes the image of $a$ in $A/P$.  Then we view $A$ as
  a $C_{0}(X)$-module via
  \begin{equation*}
    \phi\cdot a=\Phi_{A}(\phi\circ \sigma)a.
  \end{equation*}

  Also recall that
  \begin{equation*}
    I_{x}=\clsp\set{\phi\cdot a:\text{$a\in A$ and $\phi\in J_{x}$}},
  \end{equation*}
  where $J_{x}=\set{\phi\in C_{0}(X):\phi(x)=0}$.  Similarly, if
  $q:A\to A/K$ is the quotient map, then
  \begin{equation*}
    \oI_{x}=\clsp\set{\phi\cdot q(a):\text{$a\in A$ and
        $\phi\in J_{x}$}}.
  \end{equation*}
  Notice that if $P\supset K$, then
  \begin{equation*}
    \phi\cdot q(a)(P/K)= \phi(\sigma(P))q(a)(P/K) = q(\phi(\sigma(P))a)(P/K).
  \end{equation*}
  On the other hand, if $P\supset K$, then the natural isomorphism of
  $(A/K)/(P/K)$ with $A/P$ carries $q(a)(P/K)$ to $a(P)$.  It follows
  that
  \begin{equation*}
    q(\phi(\sigma(P))a)(P/K)=q(\phi\cdot a)(P/K).
  \end{equation*}
  Thus,
  \begin{equation*}
    \oI_{x}=\clsp\set{q(\phi\cdot a):\text{$a\in A$ and
        $\phi\in J_{x}$}} = q(I_{x})=(I_{x}+K)/K.
  \end{equation*}

  For the final statement, notice that
  \begin{equation*}
    (A/K)(x) = (A/K)/\oI_{x}
    = (A/K)/((I_{x}+K)/K) \cong A/(I_{x}+K). \qedhere
  \end{equation*}
\end{proof}

We will make use of the following remark in
\S\ref{sec:invariant-ideals}.

\begin{remark}
  \label{rem-cohen}
  Note that if $X$ is a $A\sme B$-\ib\ and if $J$ is an ideal in $A$,
  then $Y:=\clsp\set{a\cdot x:\text{$a\in J$ and $x\in X$}}$ is a
  nondegenerate $J$-module.  Then, employing the Cohen Factorization
  Theorem (\cite{rw:morita}*{Proposition~2.33}),
  \begin{align*}
    Y&= \set{a\cdot y:\text{$a\in J$ and $y\in Y$}}
    \subset \set{a\cdot x:\text{$a\in J$ and $x\in X$}} \\
    &\subset \clsp\set{a\cdot x:\text{$a\in J$ and $x\in X$}} =Y.
  \end{align*}
  Consequently, $Y=\set{a\cdot x:\text{$a\in J$ and $x\in X$}}$, and
  will routinely write $J\cdot X$ in place of $Y$.  Similarly, we'll
  write $X\cdot I$ for the corresponding $A\sme B$-submodule when $I$
  is an ideal in $B$.
\end{remark}

\section{The $G$-action on $\prima$}
\label{sec:boldmath-g-action}

Now we want to see that $\prima$ admits a $G$-action.  The key is to
recall that for each $x\in G$, $B(x)$ is a $A\bigl(r(x)\bigr) \sme
A\bigl(s(x)\bigr)$-\ib.  Thus by \cite{rw:morita}*{Corollary~3.33},
the \emph{Rieffel correspondence} defines a homeomorphism
\begin{equation*}
  h_{x}:\Prim A\bigl(s(x)\bigr) \to \Prim A\bigl(r(x)\bigr)
\end{equation*}
where $h_{x}$ is the restriction of $\ibind{B(x)}$ to $\Prim
A\bigl(s(x)\bigr)$.\footnote{Recall that if $X$ is an $A\sme B$-\ib,
  then $\ibind{X}$ is a continuous map from $\mathcal{I}(B)$ to
  $\mathcal{I}(A)$ characterized by $\ibind{X}(\ker
  L)=\ker\bigl(\ibind{X}(L)\bigr)$ for representations $L$ of $B$ (see
  \cite{rw:morita}*{Proposition~3.24}).\label{fn-1}} We will use the
convenient facts that $h_{x^{-1}}$ is the inverse to $h_{x}$, and that
$h_{x}$ is containment preserving \cite{rw:morita}*{Corollary~3.31}.

Then we can define
\begin{equation}
  \label{eq:1}
  x\cdot \bigl(s(x),P\bigr) := \bigl(r(x),h_{x}(P)\bigr).
\end{equation}
Suppose that $(x,y)\in G^{(2)}$ and that $L$ is a representation of
$A\bigl(s(x)\bigr)$.  It is not hard to check that
\begin{equation*}
  \ibind{B(x)}\bigl(\ibind{B(y)}(L)\bigr)\cong
  \ibind{\bigl(B(x)\tensor _{A(s(x))}B(y)\bigr)}(L).
\end{equation*}
Since $B(x)\tensor_{A(s(x))}B(y)$ and $B(xy)$ are isomorphic as \ib s
by \cite{muhwil:dm08}*{Lemma~1.2}, it follows from
\cite{rw:morita}*{Proposition~3.24} that
\begin{equation*}
  h_{x}\circ h_{y}=h_{xy}.
\end{equation*}
Since $h_{u}=\id_{\Prim A(u)}$, it follows that \eqref{eq:1} defines
an action of $G$ on $\prima$.
\begin{lemma}
  \label{lem-char-hx}
  Suppose that $x\in G$ and that $P\in \Prim A\bigl(s(x)\bigr)$.  Then
  \begin{equation*}
    h_{x}(P)=\clsp\set{adb^{*}:\text{$a,b\in B(x)$ and $d\in P$}}.
  \end{equation*}
\end{lemma}
\begin{proof}
  By the axioms for Fell bundles, the $A\bigl(r(x)\bigr)$-valued inner
  product on the \ib{} $B(x)$ is given by $\lip *<a,b>=ab^{*}$.  Then
  we can apply \cite{rw:morita}*{Proposition~3.24} to check that
  \begin{align*}
    h_{x}(P)&:= \ibind{B(x)}(P) \\
    &= \clsp\set{\lip *<ad,b>:\text{$a,b\in B(x)$ and $d\in P$}} \\
    &= \clsp\set{adb^{*}:\text{$a,b\in B(x)$ and $d\in P$} }\qedhere
  \end{align*}
\end{proof}

\begin{remark}
  \label{rem-same-def}
  We want to see that the $G$-action given by \eqref{eq:1} is the same
  as the usual one when $p:\B\to G$ is the Fell bundle associated to a
  dynamical system $(D,G,\alpha)$ (as in
  \cite{muhwil:dm08}*{Example~2.1}).  To start with, assume that $G$
  is a group.  In this case $\B$ is the trivial bundle $D\times G$ and
  the multiplication in $\B$ is given by
  $(a,s)(b,t)=(a\alpha_{s}(g),st)$.  Although it is tempting to simply
  identify $D$ with $A(e)$, it is useful to distinguish the two.  In
  particular, if $P\in\Prim A(e)$, then there is an ideal $\tp\in\Prim
  D$ such that $P=\set{(d,e):d\in \tp}$.  Then Lemma~\ref{lem-char-hx}
  implies that
  \begin{align*}
    h_{s}(P)&=\clsp\set{(a,s)(d,e)(\alpha_{s^{-1}}(b^{*})
      ,s^{-1}):\text{$a,b\in D$ and
        $d\in \tp$}} \\
    &= \clsp\set{(a\alpha_{s}(d)b^{*},e): \text{$a,b\in D$ and
        $d\in \tp$}} \\
    &= (\alpha_{s}(\tp),e),
  \end{align*}
  and we recover the usual thing.

  If $G$ is a groupoid and $D=\sa_{0}(\go;\D)$, then $\B$ is the
  pull-back $r^{*}\D$, and the multiplication in $\B$ is given by
  \begin{equation*}
    (a,x)(b,y)=\bigl(a\alpha_{x}(b),xy\bigr).
  \end{equation*}
  Again, it is useful to distinguish the fibre $D(u)$ and its image
  $A(u)=\set{(d,u):d\in D(u)}$ in $\B$.  Now, invoking
  Lemma~\ref{lem-char-hx}, if $P\in \Prim A\bigl(s(x)\bigr)$ and
  $\overline P$ is the ideal in $D\bigl(s(x)\bigr)$ such that
  $P=\set{(d,s(x)):d\in \overline P}$, then
  \begin{align*}
    h_{x}(P)&= \clsp\set{(a,x)(d,s(x))(\alpha_{x^{-1}}(b^{*}),x^{-1}):
      \text{$a,b\in A\bigl(r(x)\bigr)$
        and $d\in\overline P$}} \\
    &= \clsp\set{\bigl(a\alpha_{x}(d)b^{*},r(x)\bigr):\text{$a,b\in
        A\bigl(r(x)\bigr)$
        and $d\in\overline P$} }\\
    &= \bigl(\alpha_{x}(\overline P),r(x)\bigr) \\
    &= \alpha_{x}(P).
  \end{align*}
  Therefore we recover the usual $G$-action on $\prima$ in this case
  as well.
\end{remark}

\begin{remark}[Viewing $h_{x}$ as a map on ideals]
  \label{rem-notation}
  To ease the notational burden, we will also write $h_{x}$ for the
  map $\ibind{B(x)}:\mathcal{I}\bigl(A(s(x)\bigr)\to
  \mathcal{I}\bigl(A(r(x)\bigr)$.  In view of
  \cite{rw:morita}*{Theorem~3.29}, it is still the case that
  $h_{x^{-1}}$ is the inverse to $h_{x}$, and of course, $h_{x}$ is
  still containment preserving.  Note that if $J$ is an ideal in $A$,
  then its image $q_{u}(J)$ in $A(u)$ is $\set{c(u):c\in J}$.  In
  particular, repeating the proof of Lemma~\ref{lem-char-hx}, we see
  that
  \begin{equation}
    \label{eq:8}
    h_{x}\bigl(q_{s(x)}(J)\bigr) = \clsp \set{ac\bigl(s(x)\bigr)b^{*}:
      \text{$a,b\in B(x)$ and $c\in J$}}.
  \end{equation}
\end{remark}

\begin{prop}
  \label{prop-continuous}
  If $p:\B\to G$ is a Fell bundle and $A$ is the associated
  \cs-algebra over $\go$, then the $G$-action on $\prima$ defined by
  \eqref{eq:1} is continuous and $\prima$ is a $G$-space.
\end{prop}
\begin{proof}
  At this point, we just need to show that if $(x_{i},\tp_{i})\to
  (x_{0},\tp_{0})$ in $G*\prima$, then $x_{i}\cdot \tp_{i}\to
  x_{0}\cdot \tp_{0}$ in $\prima$. For convenience, let
  $\tp_{i}=\bigl(s(x_{i}),P_{i}\bigr)$.

  Suppose that $x_{0}\cdot
  \tp_{0}\in\OO_{J}=\set{K\in\prima:K\not\supset J}$. Then it will
  suffice to see that $x_{i}\cdot\tp_{i}$ is eventually in $\OO_{J}$.
  Suppose not.  Then we can pass to a subnet, relabel, as assume that
  for all $i\not=0$, we have
  \begin{equation}
    \label{eq:3}
    x_{i}\cdot \bigl(s(x_{i}),P_{i}\bigr)\supset J.
  \end{equation}
  Since $x_{i}\cdot \bigl(s(x_{i}),P_{i}\bigr)
  =\bigl(r(x_{i}),h_{x_{i}}(P_{i})\bigr)$, \eqref{eq:3} implies that
  $h_{x_{i}}(P_{i})\supset q_{r(x_{i})}(J)$.  Since the inverse of
  $h_{x_{i}}$ is $h_{x_{i}^{-1}}$, we have
  \begin{equation}
    \label{eq:4}
    P_{i }\supset
    h_{x_{i}^{-1}}\bigl(q_{r(x_{i})}(J)\bigr)\quad\text{for all
      $i\not=0$.} 
  \end{equation}
  I claim it will suffice to see that \eqref{eq:4} holds for $i=0$.
  To see this, notice that if \eqref{eq:4} holds for $i=0$, then
  \begin{equation*}
    h_{x_{0}}(P_{0})\supset q_{r(x_{0})}(J),
  \end{equation*}
  and since
  $\bigl(r(x),h_{x}(P)\bigr)=q_{r(x)}^{-1}\bigl(h_{x}(P)\bigr)$, this
  implies
  \begin{equation*}
    x_{0}\cdot
    \bigl(s(x_{0}),P_{0}\bigr)=\bigl(r(x_{0}),h_{x_{0}}(P_{0})\bigr) \supset
    q_{r(x_{0})}^{-1}\bigl( q_{r(x_{0})}(J)\bigr) \supset J.
  \end{equation*}
  But this contradicts the assumption that $x_{0}\cdot \tp_{0}\in
  \OO_{J}$ and will complete the proof.

  To establish the claim, we notice that when $i=0$ the right-hand
  side of \eqref{eq:4} is given by
  \begin{equation*}
    \clsp\set{a^{*}c\bigl(r(x_{0})\bigr)b:\text{$a,b\in B(x)$ and $c\in
        J$}}, 
  \end{equation*}
  where we have invoked the fact that $B(x^{-1})=B(x)^{*}$ and used
  \eqref{eq:8} from Remark~\ref{rem-notation}.  Thus it will suffice
  to show that for any $c\in J$ and $a,b\in B(x_{0})$, we have
  $a^{*}c\bigl(r(x_{0})\bigr)b\in P_{0}$.

  Since we always assume that Fell bundles have enough
  sections,\footnote{This is actually automatic.  See the comments on
    page~51 of \cite{muhwil:dm08}*{Appendix~A}.} there are $f,g\in
  \sa_{c}(G;\B)$ such that $f(x_{0})=a$ and $g(x_{0})=b$. Then we can
  form a section $\xi\in \sa_{c}(G;r^{*}\A)$ in the \cs-algebra
  $C:=\sa_{0}(G;r^{*}\A)$ given by
  \begin{equation*}
    \xi(x):= f(x)^{*}c\bigl(r(x)\bigr)g(x).
  \end{equation*}

  Notice that if $\pi$ is an irreducible representation of $A$ with
  $\sigma(\ker \pi)=u$, then there is an associated irreducible
  representation $\bar\pi$ of $A(u)$ such that $\pi=\bar\pi\circ
  q_{u}$.  If $x\in G$ and $r(x)=u$, then we get an irreducible
  representation $[x,\pi]$ of $C$ by
  $[x,\pi](f)=\bar\pi\bigl(f(x)\bigr)$.  Furthermore, by
  \cite{raewil:tams85}*{Proposition~1.3~and Lemma~1.2}, the spectrum
  of $C$ is homeomorphic to
  \begin{equation*}
    \set{[x,\pi]\in G\times\hat A:r(x)=\sigma(\ker\pi)}.
  \end{equation*}

  Now let $\pi_{i}$ be an irreducible representation of $A$ with
  kernel $\tp_{i}$.  Then, since the topology on $\hat A$ is pulled
  back from the topology on $\Prim A$, $\pi_{i}\to\pi_{0}$ in $\hat
  A$.  Consequently,
  \begin{equation}
    \label{eq:7}
    [x_{i}^{-1},\pi_{i}]\to[x_{0}^{-1},\pi_{0}]\quad\text{in $\hat C$}.
  \end{equation}
  By construction,
  \begin{equation*}
    \xi(x_{i})\in\clsp\set{a^{*}db:\text{$a,b\in B(x_{i})$ and $d\in
        q_{r(x_{i})}(J)$}} =h_{x_{i}^{-1}}\bigl(q_{r(x_{i})}(J)\bigr).
  \end{equation*}
  Hence $\xi(x_{i})\in P_{i}=\ker\bar\pi_{i}$ by \eqref{eq:4}.
  Therefore
  \begin{equation*}
    [x_{i}^{-1},\pi_{i}](\xi)=0\quad\text{for all $i\not=0$.}
  \end{equation*}
  Therefore $[x_{0}^{-1},\pi_{0}](\xi)=0$ by \eqref{eq:7}.  Since $c$
  was an arbitrary element of $J$, this proves that \eqref{eq:4} holds
  for $i=0$ and completes the proof.
\end{proof}

\section{Invariant Ideals}
\label{sec:invariant-ideals}

It is a classic result in crossed products that if $(A,G,\alpha)$ is a
dynamical system and if $I$ is an
$\alpha$-invariant ideal, then there is a short exact sequence
\begin{equation*}
  \xymatrix{0\ar[r]&I\rtimes_{\alpha}G\ar[r]^{\iota\rtimes\id}&A\rtimes_{\alpha}G
  \ar[r]^-{q\rtimes\id}&(A/I)\rtimes_{\alpha^{I}}G\ar[r]&0}
\end{equation*}
(see \cite{wil:crossed}*{Proposition~3.19}).  In this section, we want
to prove a similar result for Fell bundles.  This entails some
nontrivial work.  Even to start, we need to determine what an
invariant ideal is, and which \cs-algebras correspond to $I$ and the
quotient $A/I$.

\subsection{Preliminaries}
\label{sec:pr2}

We assume that $p:\B\to G$ is a separable Fell bundle over a locally
compact Hausdorff groupoid $G$.  Let $A=\sa_{0}(\go;\B)$ be the
\cs-algebra over $\go$.  We say that an ideal $I$ in $A$ is
\emph{$G$-invariant} if the closed set
\begin{equation*}
  \operatorname{hull}(I)=\set{\tp\in\prima:\tp\supset I}
\end{equation*}
is a $G$-invariant subset of $\prima$ with respect to the $G$-action
introduced in Proposition~\ref{prop-continuous}.

Now fix an ideal $I\in\mathcal{I}(A)$.  If $q_{u}:A\to A(u)=A/J_{u}$
is the quotient map, then we let $I(u):=q_{u}(I)=(I+J_{u})/J_{u}$.
\begin{lemma}
  \label{lem-key-ix}
  Let $h_{x}:\I\bigl(A(s(x))\bigr)\to \I\bigl(A(r(x))\bigr)$ be the
  Rieffel correspondence.  If $I$ is a $G$-invariant ideal, then
  \begin{equation*}
    h_{x}\bigl(I(s(x))\bigr)=I(r(x)).
  \end{equation*}
  In particular, $B(x)\cdot I(s(x))=I(r(x))\cdot B(x)$.
\end{lemma}
\begin{remark}
  \label{rem-dots}
  We have included a \verb=\cdot= in the above notation to stress that
  $B(x)\cdot I(s(x))$ is the sub-bimodule corresponding to $I(s(x))$ in
  the \ib\ $B(x)$.  Since the right action is just given by
  multiplication in $\B$, the \verb=\cdot= can be dropped without any
  harm.  In fact, it will be critical in what follows that we are just
  dealing with multiplication in $\B$ which is an associative
  operation (when that makes sense).
\end{remark}

\begin{proof}
  If $P\in\Prim \bigl(A(s(x))\bigr)$ and $I$ is invariant, then
  \begin{align*}
    P\supset I(s(x)) &\iff \bigl(s(x),P\bigr) \supset I \\
    &\iff \bigl(r(x),h_{x}(P)\bigr) \supset I \\
    &\iff h_{x}(P) \supset I(r(x)).
  \end{align*}
  The first assertion follows.  The second assertion follows from
  \cite{rw:morita}*{Proposition~3.24}.
\end{proof}

Now we define
\begin{equation*}
  \BI:=\set{b\in \B: b^{*}b\in I(s(b))},
\end{equation*}
where, as is usual, we write $s(b)$ as a shorthand for $s(p(b))$.
\begin{lemma}
  \label{lem-bsubi}
  If $I$ is an ideal in $A$, then
  \begin{equation*}
    \BI=\set{b\in\B:\text{$a^{*}b\in I(s(b))$ for all $a\in B(p(b))$}}.
  \end{equation*}
  In particular, $b\in \BI$ implies that $b\in B(p(b))\cdot I(s(b))$.
\end{lemma}
\begin{remark}
  \label{rem-always-ib}
  Note that $B(x)\cdot I(s(b))$ is always an \ib\ between $I(s(b))$
  and the ideal of $A(r(x))$ corresponding to $I(s(b ))$ under the
  Rieffel correspondence \cite{rw:morita}*{Proposition~3.25}.
\end{remark}

\begin{proof}
  Since $B(p(b))$ is a right Hilbert $A(s(b))$-module, this result
  follows from \cite{rw:morita}*{Lemma~3.23}.
\end{proof}

\begin{prop}
  \label{prop-bi}
  Suppose that $p:\B\to G$ is a Fell bundle over a locally compact
  Hausdorff groupoid $G$ and that $I$ is an ideal in the \cs-algebra
  $A=\sa_{0}(\go;\B)$. Let $\BI=\set{b\in\B:b^{*}b\in I(s(b))}$ and
  $p_{I}=p\restr {\BI}$.  Then $p_{I}:\BI\to G$ is an \usc-Banach
  bundle with fibres $\BI(x)=B(x)\cdot I(s(b))$.\footnote{It is
    possible --- even likely --- that some of the fibres $\BI(x)$ are
    the zero space.}  If $I$ is $G$-invariant, then $\BI$ is a Fell
  bundle with the operations inherited from $\B$.
\end{prop}
\begin{proof}
  It is fairly straightforward to check that $p_{I}:\BI\to G$ is an
  \usc-Banach bundle with the exception of verifying that $p_{I}$ is
  open.\footnote{Keep in mind that the restriction of an open map need
    not be open.}  To prove that, we'll use
  \cite{wil:crossed}*{Proposition~1.15}.  Suppose that $b\in\BI$ and
  that $x_{i}\to p(b)$.  It will suffice to find, after passing to a
  subsequence and relabeling, elements $b_{i}\in \BI$ such that
  $b_{i}\to b$ and $p(b_{i})=x_{i}$.

  Since $\BI(p(b))=B(p(b))\cdot I(s(b))$, in view of
  Remark~\ref{rem-cohen} we can suppose that
  $b=b'\cdot a(s(b))$ where $a\in I$ (and we write $a(u)$ for the
  image of $a$ in $I(u)$). Since $p$ is open, we can pass to a
  subsequence, relabel, and find $b_{i}'\to b'$ such that
  $p(b_{i}')=x_{i}$.  But $a(s(x_{i}))\to a(s(b))$ in $\B$, so the
  continuity of multiplication implies that $b_{i}' \cdot
  a(s(x_{i}))\to b'\cdot a(s(b))=b$.  But $b_{i}:=b_{i}'\cdot
  a(s(x_{i}))\in\BI$. This completes the proof that $p_{I}$ is open
  and the proof that $p_{I}$ is an \usc-Banach bundle.

  Now we assume that $I$ is $G$-invariant. Suppose that
  $(b,b')\in\BI^{(2)}$.  Then $b\in B(p(b))\cdot I(s(b))$, $b'\in
  B(p(b'))\cdot I(s(b'))$ and $s(b)=r(b')$.  Thus using
  Lemma~\ref{lem-key-ix},
  \begin{align*}
    bb'&\in B(p(b))I(s(b))B(p(b'))I(s(b')) \\
    & \subset B(p(b))B(p(b'))I(s(b'))^{2} \\
    & \subset B(p(bb')I(s(bb')).
  \end{align*}
  Thus $bb'\in \BI$ and $\BI$ is closed under multiplication.
  Similarly, $b^{*}\in I(s(b))B(p(b)^{-1})= B(p(b^{*}))I(s(b^{*}))$,
  and $b^{*}\in \BI$.

  Therefore axioms (a)--(c) of \cite{muhwil:dm08}*{Definition~1.1} are
  clearly satisfied.  Since $\BI(u)=I(u)$, axiom~(d) is satisfied, and
  \cite{rw:morita}*{Proposition~3.25} together with
  Lemma~\ref{lem-key-ix} imply that $\BI(x)$ is an $I(r(x))\sme
  I(s(x))$-\ib.  Thus (e) holds as well.
\end{proof}

{\emergencystretch=50pt Now we let $\bqi(x)$ be the quotient module
  $B(x)/\bi(x)$.  If $J$ is the ideal of $A(r(x))$ corresponding to
  $I(s(x))$ under the Rieffel correspondence (so that $J=I(r(x))$ if
  $I$ is invariant by Lemma~\ref{lem-key-ix}), then $\bqi(x)$ is  a
  $A(r(x))/J\sme A(s(x))/I(s(x))$-\ib\ 
  by \cite{wil:crossed}*{Proposition~3.25}.
  Let
  \begin{equation*}
    \BqI:=\coprod_{x\in G}\bqi(x)
  \end{equation*}
  and form the bundle $p^{I}:\BqI\to G$.  If $f\in\sa_{c}(G;\B)$, then
  we define a section $q(f)$ of $p^{I}:\BqI\to G$ by $q(f)(x)=[f(x)]$,
  where as usual, $[b]$ denotes the class of $b\in B(x)$ in
  $\bqi(x)$.\par}

\begin{prop}
  \label{prop-bqi}
  Suppose that $p:\B\to G$ is a Fell bundle over a locally compact
  Hausdorff groupoid $G$ and that $I$ is an ideal in the \cs-algebra
  $A=\sa_{0}(\go;\B)$.  Then $\BqI$ has a topology making
  $p^{I}:\BqI\to G$ an \usc-Banach bundle such that
  $\Gamma:=\set{q(f):f\in\sa_{c}(G;\B)}$ is a dense subspace of
  $\sa_{c}(G;\BqI)$ in the inductive limit topology.  If $I$ is
  $G$-invariant, then $p^{I}:\BqI\to G$ is a Fell bundle with the
  operations induced from $\B$.
\end{prop}

\begin{proof}[Proof of Proposition~\ref{prop-bqi}]
  Recall that $\bqi(x)$ is a quotient imprimitivity module, which is
  in particular a right Hilbert $A(s(x))/I(s(x))$-module, where the
  Hilbert module operations are induced from the Fell bundle
  operations on $\B$.  Therefore,
  \begin{equation*}
    \|q(f)(x)\|^{2} = \|\bar q (f(x)^{*}f(x))\|,
  \end{equation*}
  where $\bar q:A\to A/I(s(x))$ is the quotient map.  But
  Lemma~\ref{lem-cox} implies that $A/I$ is a $C_{0}(\go)$-algebra
  with fibre over $u$, $(A/I)(u)$, naturally isomorphic to
  $A(u)/I(u)$.  Therefore,
  \begin{equation*}
    \|q(f)(x)\|^{2}= \|q_{I}(f(x)^{*}f(x))(s(x))\|,
  \end{equation*}
  where $q_{I}:A\to A/I$ is the quotient map.  Therefore, $x\mapsto
  \|q(f)(x)\|$ is the composition of the continuous map $x\mapsto
  f(x)^{*}f(x)$ from $G$ to $A$ with the upper semicontinuous map
  $q_{I}(a) \mapsto \|q_{I}(a)(u)\|$ coming from the
  $C_{0}(\go)$-algebra structure
  (\cite{wil:crossed}*{Proposition~C.10}).  Therefore $x\mapsto
  \|q(f)(x)\|$ is upper semicontinuous and we can apply
  Theorem~\ref{thm-build-bundle} to give $\BqI$ a topology such that
  $p^{I}:\BqI\to G$ is an \usc-Banach bundle such that $\Gamma:=
  \set{q(f):g\in\sa_{c}(G;\B)} \subset \sa_{c}(G;\BqI)$.  Furthermore,
  since $\Gamma$ is a $C_{0}(G)$-module, it follows from
  \cite{muhwil:dm08}*{Lemma~A.4} that $\Gamma$ is dense in
  $\sa_{c}(G;\BqI)$ in the inductive limit topology. This establishes
  all but the last assertion.

  Now assume that $I$ is $G$-invariant.  Let $(x,y)\in G^{(2)}$.
  Suppose that $a,a'\in B(x)$, that $b,b'\in B(y)$ and that
  $a'=a+a''$, $b'=b+b''$ with $a''\in \bi(x)$ and $b''\in\bi(y)$.
  Since $s(x)=r(y)$, we can apply Lemma~\ref{lem-key-ix} to calculate
  as follows:
  \begin{align*}
    a'b'&= ab+a''b+ab''+a''b'' \\
    &\in ab +B(x)I(r(y))B(y) + B(x)B(y)I(s(y)) +
    B(x)I(r(y))B(y)I(s(y)) \\ 
    &\subset ab + B(xy)I(s(y)).
  \end{align*}
  Therefore, we get a well-defined multiplication on $(\BqI)^{(2)}$.

  To see that multiplication is continuous from $(\BqI)^{(2)}$ to
  $\BqI$ we first want to establish that the quotient map $a\mapsto
  [a]$ is continuous from $\B$ to $\BqI$.  To this end, suppose that
  $a_{i}\to a$ in $\B$.  Let $f\in \gcgb$ be such that $f(p(a))=a$.
  Then
  \begin{equation*}
    \|f(p(a_{i}))-a_{i}\|\to 0.
  \end{equation*}
  But then
  \begin{equation*}
    \|q(f)(p(a_{i}))-[a_{i}]\|\to 0.
  \end{equation*}
  Since $q(f)\in\sa_{c}(G;\BqI)$, we have $[a_{i}]\to[a]$ by
  \cite{wil:crossed}*{Proposition~C.20}.

  Now suppose that $([a_{i}],[b_{i}])\to ([a],[b])$ in $(\BqI)^{(2)}$
  with 
  $\bigl(p^{I}([a_{i}]),p^{I}([b_{i}])\bigr)=(x_{i},y_{i})$ 
converging to $(x,y)$ in $G^{(2)}$.  Let $f,g\in\gcgb$ be such that
  $q(f)(x)=[a]$ and $q(g)(x)=[b]$.  By the above observation, we have
  \begin{equation*}
    q(f)(x_{i})q(g)(y_{i})\to [a][b].
  \end{equation*}
  However, using Lemma~\ref{lem-norm-ineq}, we see easily that
  \begin{equation*}
    \|q(f)(x_{i})q(g)(y_{i})-[a_{i}][b_{i}]\|\to 0.
  \end{equation*}
  Therefore $[a_{i}][b_{i}]\to [a][b]$ by
  \cite{wil:crossed}*{Proposition~C.20}, and multiplication is
  continuous.  The argument that $[a]\mapsto [a^{*}]$ is a well
  defined and continuous involution is similar and we omit the
  details.

  It is now straightforward to see that $p^{I}:\BqI\to G$ is a Fell
  bundle as claimed: properties (a)--(c) of
  \cite{muhwil:dm08}*{Definition~1.1} are clearly satisfied.  On the
  other hand, if $u\in\go$, then $\bqi(u)$ is the \cs-algebra
  $A(u)/I(u)=(A/I)(u)$, while if $x\in G$, then $\bqi(x)$ is a
  $(A/I)(r(x))\sme (A/I)(s(x))$-\ib\ with respect to the quotient
  operations.  Therefore properties (d)~and (e) are also satisfied.
\end{proof}

\subsection{The Exact Sequence}
\label{sec:exact-sequence}

In this section, we are always assuming that $I$ is a
\emph{$G$-invariant ideal} 
in the \cs-algebra $A=\sa_{0}(\go;\B)$ sitting over $\go$ in a Fell
bundle $p:\B\to G$ over a locally compact Hausdorff groupoid $G$.
Of course we will use the properties of $\BI$ and $\BqI$ from
Propositions \ref{prop-bi}~and \ref{prop-bqi}.

We clearly have an injective $*$-homomorphism
\begin{equation*}
  \iota:\sa_{c}(G;\BI)\to \gcgb
\end{equation*}
given by inclusion.  Also $q\mapsto q(f)$ is  a $*$-homomorphism
\begin{equation*}
  q:\gcgb\to \sa_{c}(G;\BqI).
\end{equation*}
Proposition~\ref{prop-bqi} implies that $q(\gcgb)$ is dense in
$\sa_{c}(G;\BqI)$ in the inductive limit topology.  Therefore $q$ has
dense range when viewed a map into $\cs(G,\BqI)$.
\begin{lemma}
  \label{lem-i}
  The map $\iota$ extends to an isomorphism of $\cs(G,\BI)$ onto an
  ideal $\Ex(I):=\overline{\iota(\sa_{c}(G;\BI))}$ in $\gcgb$.
\end{lemma}
\begin{proof}
  Suppose that $f\in \gcgbi$ and that $g\in\gcgb$.  Then, using
  Lemma~\ref{lem-key-ix}, 
  $f(x)g(x^{-1}y) \in B(x)I(s(x))B(x^{-1}y) = B(y)I(s(y))=\bi(y)$, and
  $f^{*}(x)\in I(s(x))B(x^{-1})=B(x^{-1})I(r(x))=\bi(x^{-1})$.  It now
  follows easily that $\Ex(I)$ is an ideal in $\cs(G,\B)$.  We just
  need to see that $\iota$ is isometric for the universal norms.  Let
  $L$ be an \emph{irreducible} representation of $\cs(G,\B)$.  Then
  either $L(\iota(\gcgbi))=\set0$ or 
  $L$ defines  a representation, $L'$ of $\gcgbi$ in the sense of
  \cite{muhwil:dm08}*{Definition~4.7} (which is
  nondegenerate because $L$ is irreducible).  Then
  \begin{equation*}
    \|L(\iota(f))\|=\|L'(f)\|\le \|f\|.
  \end{equation*}
Since $\|L(\iota(f))\|\le \|f\|$ holds for all irreducible
representations, we have
\begin{equation*}
  \|\iota(f)\|\le\|f\|.
\end{equation*}

Now let $L'$ be a faithful representation of $\cs(G,\BI)$ on $\H$.
Let $\H_{0}$ be the dense subspace
\begin{equation*}
  \H_{0}=\operatorname{span}\set{L'(f)h:\text{$f\in \gcgbi$ and
      $h\in\H$}}. 
\end{equation*}
Suppose that $f_{1},\dots,f_{k}\in\gcgbi$ and that
$h_{1},\dots,h_{k}\in\H$ are such that
\begin{equation}
  \label{eq:16}
  \sum_{i=1}^{k}L'(f)h_{i}=0.
\end{equation}
Let $g\in\gcgb$.  Let $\set{e_{j}}$ be an approximate identity for
$\gcgbi$ in the inductive limit topology (see
\cite{muhwil:dm08}*{Proposition~5.1}).  Then, since convolution is
continuous in the inductive limit topology,  $g*e_{j}*f_{i}\to
g*f_{i}$ in the inductive limit topology.  Thus, using \eqref{eq:16},
\begin{equation*}
  \sum_{i=1}^{k}L'(g*f_{i})h_{i} = \lim_{j} \sum_{i=1}^{k}
  L'(g*e_{j}*f_{i})h_{i} 
= \lim_{j} L'(g*e_{j})\sum_{i=1}^{k}L'(f_{i})h_{i} 
=0.
\end{equation*}
Thus we get a well-defined
homomorphism $L$ from $\gcgb$ to the linear operators
$\operatorname{Lin}(\H_{0})$ on $\H_{0}$ characterized by
\begin{equation*}
  L(g)L'(f)h:=L'(g*f)h.
\end{equation*}
(Notice that $L(\iota(f))$ is just the restriction of $L'(f)$ to
$\H_{0}$.) We claim that $L$ is what we called a
\emph{pre-representation} of $\B$ in
\cite{muhwil:dm08}*{Definition~4.1}.  To see this, notice that if
$g_{i}\to g$ in the inductive limit topology on $\gcgb$, then
$g_{i}*f\to g*f$ in the inductive limit topology on $\gcgbi$ for any
$f\in \gcgbi$.  Therefore
\begin{equation*}
  g\mapsto \ip(L(g)v|w)
\end{equation*}
is continuous in the inductive limit topology for all $v,w\in
\H_{0}$.  Therefore condition~(a) of
\cite{muhwil:dm08}*{Definition~4.1} is satisfied.  To verify
condition~(b), just note that
\begin{align*}
  \bip(L(g)L'(f)h|{L'(f')h'}) &= \bip(L'(g*f)h|{L'(f')h'}) \\
\intertext{which, since $L'$ is a $*$-homomorphism and since
$(g*f)^{*}=f^{*}*g^{*}$, is}
&= \bip(h|{L'(f^{*}*g^{*}*f')h'})\\
&= \bip(L'(f)h|{L(g^{*})L'(f')h'}).
\end{align*}
Lastly, condition~(c) follows easily from the existence of an
approximate identity for $\gcgb$ in the inductive limit topology.
Now the Disintegration Theorem (\cite{muhwil:dm08}*{Theorem~4.13})
implies that $L$ can be extended to a bounded representation of
$\cs(G,\B)$.  But then
\begin{equation*}
  \|f\|=\|L'(f)\|=\|L(\iota(f))\|\le \|\iota(f)\|.
\end{equation*}
Thus $\iota$ is isometric as claimed.
\end{proof}

\begin{lemma}
  \label{lem-q}
  The $*$-homomorphism $q$ is bounded and extends to a surjective
  homomorphism $q$ of $\cs(G,\B)$ onto $\cs(G,\BqI)$.
\end{lemma}
\begin{proof}
  Let $L'$ be a faithful representation of $\cs(G,\BqI)$.  By the
  Disintegration Theorem
  \cite{muhwil:dm08}*{Theorem~4.13}, we can assume that $L'$ is the
  integrated form of a strict representation $(\mu,\go*\HH,\hat\pi')$,
  where $\mu$ is a quasi-invariant measure, $\go*\HH$ is a Borel
  Hilbert bundle and $\hat\pi'$ is a Borel $*$-functor on $\BqI$.
  Thus $\hat\pi'([b])=\bigl(r(b),\bar\pi([b]),s(b)\bigr)$ for a
  bounded operator $\bar\pi([b]):\H(s(b))\to\H(r(b))$ with
  $\|\bar\pi([b])\|\le\|[b]\|$. Then we can define
  $\hat\pi(b)=\bigl(r(b),\pi(b),s(b)\bigr)$, where
  $\pi(b):=\bar\pi([b])$.  Then $(\mu,\go*\HH,\hat\pi)$ is a strict
  representation of $\B$, and its integrated form, $L$, satisfies
  $L=L'\circ q$.  In particular,
  \begin{equation*}
    \|q(f)\|=\|L'(q(f))\|=\|L(f)\|\le \|f\|.
  \end{equation*}
Hence $q$ is norm decreasing on $\gcgb$. 
\end{proof}

\begin{thm}
  \label{thm-main-ses}
  Suppose that $p:\B\to G$ is a Fell bundle over a locally compact
  Hausdorff groupoid $G$.  Let $A=\sa_{0}(\go;\B)$ be the
  \cs-algebra over $\go$ and suppose that $I$ is a $G$-invariant ideal
  in $A$.  Let $\BI$ and $\BqI$ be the Fell bundles described above.
  Then 
  \begin{equation*}
    \xymatrix{0\ar[r]&\cs(G,\BI)\ar[r]^{\iota}&\cs(G,\B)\ar[r]^{q}&\cs
    (G,\BqI)\ar[r]&0}
  \end{equation*}
is a short exact sequence of \cs-algebras.
\end{thm}

In view of Lemmas~\ref{lem-i} and \ref{lem-q}, it will suffice to see
that $\ker q =\Ex(I)$.  This will require some work, and we start with
some preliminary comments.

Let $a\in\tA$ and $f\in\gcgb$.  Define $\ia(a)f\in\gcgb$ by
\begin{equation*}
  (\ia(a)f)(x)=a(r(x)) f(x).
\end{equation*}
Now view $\gcgb$ as a dense subspace of the $\cs(G,\B)$ viewed as a
right Hilbert module over itself with respect to the inner product
\begin{equation*}
  \rip *<f,g>:=f^{*}*g\quad\text{for $f,g\in\gcgb$.}
\end{equation*}
Then it is easy to check that
\begin{equation*}
  \rip*<\ia(a)f,g>=\rip*<f,\ia(a^{*})g>.
\end{equation*}
Then if $\|a\|^{2}1_{A}-a^{*}a=c^{*}c$, we have
\begin{equation*}
  \|a\|^{2}\rip*<f,f>-\rip*
  <\ia(a)f,{\ia(a)f}>=\rip*<\ia(c)f,{\ia(c)f}>\ge 0.
\end{equation*}
Therefore $\ia$ is bounded and extends to a homomorphism $\ia:A\to
M(\cs(G,\B))$.

Let $L$ be a representation of $\cs(G,\B)$ on $\H$.  In view of
\cite{muhwil:dm08}*{Theorem~4.13}, we can assume that $L$ is the
integrated form of a strict representation $(\mu,\go*\HH,\hat\pi)$,
where $\mu$ is a quasi-invariant measure, $\H=\go*\HH$ is a Borel
Hilbert bundle and $\hat\pi$ is a Borel $*$-functor with
$\hat\pi(b)=\bigl(r(b),\pi(b),s(b)\bigr)$ for an operator
$\pi(b):B(\H(s(b)))\to B(\H(r(b)))$ with $\|\pi(b)\|\le\|b\|$.  (We
will often write $\pi(b)$ for both the operator and the corresponding
element of $\operatorname{End}(\go*\HH)$.)

Then $L$ defines a representation $\pi_{L}$ of $A$ on $\H$ via
composition with $\ia:A\to M(\cs(G,\B))$.  It is not hard to check
that if $h,k\in L^{2}(\go*\HH,\mu)$, then
\begin{equation*}
  \bip(\pi_{L}(a)h|k) = \int_{G}\bip(\pi(a(u))h(u)|{k(u)}) \,d\mu(u). 
\end{equation*}
In particular,
\begin{equation*}
  \pi_{L}=\int_{G}^{\oplus}\pi_{u}\,d\mu(u),
\end{equation*}
where $\pi_{u}$ is the representation of $A$ given by $a\mapsto
\pi(a(u))$. 
\begin{proof}[Proof of Theorem~\ref{thm-main-ses}]
  Clearly,
  \begin{equation}
    \label{eq:17}
    \Ex(I)\subset \ker q.
  \end{equation}
Therefore it will suffice see that given any representation $L$ of
$\cs(G,\B)$, we have either $\ker q\subset \ker L$, or
$\Ex(I)\not\subset \ker L$.

So let $L$ be a representation of $\cs(G,\B)$.  Let
$(\mu,\go*\HH,\hat\pi)$ and $\pi_{L}$ be as above.  
There are two cases to
consider.  First, $I\subset \ker \pi_{L}$, and second $I \not\subset
\ker\pi_{L}$.

\subsection*{Case \boldmath  $I\subset \ker\pi_L$}
\label{sec:case-isubset-kerpi_l}

Since $A$, and hence $I$, are separable, there is a Borel null set
$N\subset \go$ such that
\begin{equation*}
  I\subset \ker \pi_{u}\quad\text{for all $u\notin N$.}
\end{equation*}
Let $F:=\go\setminus N$.  Then if $u\in  F$, we have
$\pi(I(u))=\set 0$.   

Suppose that $s(x)\in F$.  Since $\bi(x)=B(x)\cdot I(s(x))$ is an
$I(r(x))\sme I(s(x))$-\ib, given any $b\in\bi(x)$, we have $b^{*}b\in
I(s(x))$ and $0=\pi(b^{*}b)=\pi(b)^{*}\pi(b)$.  Therefore $\pi(b)=0$,
and $\pi(bb^{*})=0$.  Since elements of the form $bb^{*}$ space a
dense subspace of $I(r(x))$, we see that $\pi(I(r(x))=\set0$.
Furthermore, since $G$ is $\sigma$-compact, we can shrink $F$ a bit it
necessary, and assume its saturation is Borel (see that last paragraph
of the proof of \cite{muhwil:dm08}*{Lemma~5.20}).  Therefore, we may
as well assume that $F$ itself is saturated.

Since $\mu$ is quasi-invariant, the restriction $G\restr F$ is
$\nu$-conull.\footnote{Recall that $\nu$ is the measure on $G$ given by
  $\int_{\go}\lambda^{u}\,d\mu(u)$.}  Furthermore, $x\in G\restr F$
and $b\in \bi(x)$ implies that $\pi(b)=0$.  Thus if $b ,b'\in B(x)$
are such that $b-b'\in\bi(x)$, then $\pi(b)=\pi(b')$.  Thus if $[b]$
is the class of $b\in B(x)$ in $\bqi(x)$, then we can define
$\bar\pi(x)\in B\bigl(\H(s(x)),\H(r(x))\bigr)$ by
\begin{equation*}
  \bar\pi([b])=\pi(b).
\end{equation*}
Now we define $\hat\pi':\BqI\to \operatorname{End}(\go*\HH)$ by
$\hat\pi'([b]) = \bigl(r(b),\bar\pi([b]),s(b)\bigr)$ where
\begin{equation*}
  \bar\pi([b])=
  \begin{cases}
    \pi(b)&\text{if $p(b)\in G\restr F$, and} \\
0&\text{otherwise.}
  \end{cases}
\end{equation*}
It is immediate that if $f\in\gcgb$, then $x\mapsto \hat\pi'(q(f)(x))$
is Borel.  Since any $\bar f\in\sa_{c}(G;\BqI)$ is the uniform limit of
sections of the form $q(f)$ with $f\in\gcgb$, it follows that
$x\mapsto \hat\pi'(\bar f(x))$ is Borel for all $\bar f\in \gcgbqi$.
Note that since $F$, and hence $N=\go\setminus F$, are saturated, $G$
is the disjoint union of the restrictions $G\restr F$ and $G\restr
N$.  Therefore, $\bar\pi([a][b])=\bar\pi([a])\bar\pi([b])$ if
$([a],[b])\in (\BqI)^{(2)}$.  Similarly, axioms (a)~and (c) of
\cite{muhwil:dm08}*{Definition~4.5} are clearly satisfied and
$\hat\pi'$ is a Borel $*$-functor on $\BqI$.  Furthermore,
$\bigl(\mu,\go*\HH,\hat\pi'\bigr)$ is a strict representation of
$\BqI$, and since $G\restr F$ is $\nu$-conull, it follows from the
Equation~(4.4) in \cite{muhwil:dm08} that
the integrated form $L'$ satisfies $L=L'\circ q$.  In
particular, $\ker q\subset \ker L$ in this case.

\subsection*{Case \boldmath $I \not\subset \ker \pi_L$}
\label{sec:case-boldmath-i}

In this case, there is a $a\in I$ such that $\pi_{L}(a)\not=0$.  Since
$L$ is nondegenerate, there is a $f\in\gcgb$ such that
$\pi_{L}(a)L(f)\not=0$.  But then $L(\ia(a)f)\not=0$.  However,
\begin{equation*}
  (\ia(a)f)(x)=a(r(x))f(x)\in I(r(x))B(x)=B(x)I(s(x))=\bi(x).
\end{equation*}
Therefore $\ia(a)f\in\Ex(I)$, and
\begin{equation*}
  \Ex(I)\not\subset \ker L
\end{equation*}
is this case.  This completes the proof.
\end{proof}

\bibliographystyle{amsxport} 
\bibliography{references-nov01}

\def\noopsort#1{}\def\cprime{$'$} \def\sp{^}
\begin{bibdiv}
\begin{biblist}

\bib{beaexe:ms01}{article}{
      author={Abadie, Beatriz},
      author={Exel, Ruy},
       title={Deformation quantization via {F}ell bundles},
        date={2001},
        ISSN={0025-5521},
     journal={Math. Scand.},
      volume={89},
      number={1},
       pages={135\ndash 160},
      review={\MR{MR1856986 (2002g:46118)}},
}

\bib{bmz}{unpublished}{
      author={Buss, A.},
      author={Meyer, R.},
      author={Zhu, C.},
       title={A higher category approach to twisted actions on
  {$C^*$}-algebras},
     address={preprint},
        date={2009},
        note={(arXiv:math.OA.0908.0455v1)},
}

\bib{busexe:xx09}{unpublished}{
      author={Buss, Alcides},
      author={Exel, Ruy},
       title={Fell bundles over inverse semigroups and twisted \'etale
  groupoids},
     address={preprint},
        date={2009},
        note={(arXiv:math.OA.0903.3388v3)},
}

\bib{bus:xx09}{unpublished}{
      author={Buss, Alcides},
       title={Integrability of dual coactions on {F}ell bundle
  {$C^*$}-algebras},
     address={preprint},
        date={2009},
        note={(arXiv:math.OA.0908.0559)},
}

\bib{dkr:ms08}{article}{
      author={Deaconu, Valentin},
      author={Kumjian, Alex},
      author={Ramazan, Birant},
       title={Fell bundles and groupoid morphisms},
        date={2008},
     journal={Math. Scan.},
      volume={103},
      number={2},
       pages={305\ndash 319},
}

\bib{dg:banach}{book}{
      author={Dupr{\'e}, Maurice~J.},
      author={Gillette, Richard~M.},
       title={Banach bundles, {B}anach modules and automorphisms of
  ${C}^*$-algebras},
   publisher={Pitman (Advanced Publishing Program)},
     address={Boston, MA},
        date={1983},
      volume={92},
        ISBN={0-273-08626-X},
      review={\MR{85j:46127}},
}

\bib{exemar:pams97}{article}{
      author={Exel, Ruy},
      author={Laca, Marcelo},
       title={Continuous {F}ell bundles associated to measurable twisted
  actions},
        date={1997},
        ISSN={0002-9939},
     journal={Proc. Amer. Math. Soc.},
      volume={125},
      number={3},
       pages={795\ndash 799},
      review={\MR{MR1353382 (97e:46073)}},
}

\bib{exe:jfram97}{article}{
      author={Exel, Ruy},
       title={Amenability for {F}ell bundles},
        date={1997},
        ISSN={0075-4102},
     journal={J. Reine Angew. Math.},
      volume={492},
       pages={41\ndash 73},
      review={\MR{MR1488064 (99a:46131)}},
}

\bib{exe:pjm00}{article}{
      author={Exel, Ruy},
       title={Partial representations and amenable {F}ell bundles over free
  groups},
        date={2000},
        ISSN={0030-8730},
     journal={Pacific J. Math.},
      volume={192},
      number={1},
       pages={39\ndash 63},
      review={\MR{MR1741030 (2001c:46109)}},
}

\bib{exe:ma02}{article}{
      author={Exel, Ruy},
       title={Exact groups and {F}ell bundles},
        date={2002},
        ISSN={0025-5831},
     journal={Math. Ann.},
      volume={323},
      number={2},
       pages={259\ndash 266},
      review={\MR{MR1913042 (2003e:46095)}},
}

\bib{fd:representations1}{book}{
      author={Fell, James M.~G.},
      author={Doran, Robert~S.},
       title={Representations of {$*$}-algebras, locally compact groups, and
  {B}anach {$*$}-algebraic bundles. {V}ol. 1},
      series={Pure and Applied Mathematics},
   publisher={Academic Press Inc.},
     address={Boston, MA},
        date={1988},
      volume={125},
        ISBN={0-12-252721-6},
        note={Basic representation theory of groups and algebras},
      review={\MR{90c:46001}},
}

\bib{fd:representations2}{book}{
      author={Fell, James M.~G.},
      author={Doran, Robert~S.},
       title={Representations of {$*$}-algebras, locally compact groups, and
  {B}anach {$*$}-algebraic bundles. {V}ol. 2},
      series={Pure and Applied Mathematics},
   publisher={Academic Press Inc.},
     address={Boston, MA},
        date={1988},
      volume={126},
        ISBN={0-12-252722-4},
        note={Banach $*$-algebraic bundles, induced representations, and the
  generalized Mackey analysis},
      review={\MR{90c:46002}},
}

\bib{fulmuh:ijm98}{article}{
      author={Fulman, Igor},
      author={Muhly, Paul~S.},
       title={Bimodules, spectra, and {F}ell bundles},
        date={1998},
        ISSN={0021-2172},
     journal={Israel J. Math.},
      volume={108},
       pages={193\ndash 215},
      review={\MR{MR1669380 (2000b:46096)}},
}

\bib{ful:hjm98}{article}{
      author={Fulman, Igor},
       title={Coordinatization for {F}ell bundle algebras},
        date={1998},
        ISSN={0362-1588},
     journal={Houston J. Math.},
      volume={24},
      number={2},
       pages={313\ndash 324},
      review={\MR{MR1690405 (2000e:46069)}},
}

\bib{hof:74}{unpublished}{
      author={Hofmann, Karl~Heinrich},
       title={Banach bundles},
        date={1974},
        note={Darmstadt Notes},
}

\bib{hof:lnm77}{incollection}{
      author={Hofmann, Karl~Heinrich},
       title={Bundles and sheaves are equivalent in the category of {B}anach
  spaces},
        date={1977},
   booktitle={{$K$}-theory and operator algebras ({P}roc. {C}onf., {U}niv.
  {G}eorgia, {A}thens, {G}a., 1975)},
      series={Lecture Notes in Math},
      volume={575},
   publisher={Springer},
     address={Berlin},
       pages={53\ndash 69},
      review={\MR{58 \#7117}},
}

\bib{ionwil:xx09b}{unpublished}{
      author={Ionescu, Marius},
      author={Williams, Dana~P.},
       title={A classic {M}orita equivalence result for {F}ell bundle
  {$C^*$}-algebras},
     address={preprint},
        date={2009},
        note={(arXiv:math.OA.0912.1125)},
}

\bib{kmqw:xx09}{unpublished}{
      author={Kaliszewski, Steven},
      author={Muhly, Paul~S.},
      author={Quigg, John},
      author={Williams, Dana~P.},
       title={Coactions and {F}ell bundles},
     address={preprint},
        date={2009},
        note={(arXiv:math.OA.0909.3259)},
}

\bib{kum:pams98}{article}{
      author={Kumjian, Alex},
       title={Fell bundles over groupoids},
        date={1998},
        ISSN={0002-9939},
     journal={Proc. Amer. Math. Soc.},
      volume={126},
      number={4},
       pages={1115\ndash 1125},
      review={\MR{MR1443836 (98i:46055)}},
}

\bib{lan:hilbert}{book}{
      author={Lance, E.~C.},
       title={Hilbert {$C\sp *$}-modules},
      series={London Mathematical Society Lecture Note Series},
   publisher={Cambridge University Press},
     address={Cambridge},
        date={1995},
      volume={210},
        ISBN={0-521-47910-X},
        note={A toolkit for operator algebraists},
      review={\MR{MR1325694 (96k:46100)}},
}

\bib{muhwil:dm08}{article}{
      author={Muhly, Paul~S.},
      author={Williams, Dana~P.},
       title={Equivalence and disintegration theorems for {F}ell bundles and
  their {$C\sp *$}-algebras},
        date={2008},
        ISSN={0012-3862},
     journal={Dissertationes Math. (Rozprawy Mat.)},
      volume={456},
       pages={1\ndash 57},
      review={\MR{MR2446021}},
}

\bib{raewil:tams85}{article}{
      author={Raeburn, Iain},
      author={Williams, Dana~P.},
       title={Pull-backs of {$C\sp \ast$}-algebras and crossed products by
  certain diagonal actions},
        date={1985},
        ISSN={0002-9947},
     journal={Trans. Amer. Math. Soc.},
      volume={287},
      number={2},
       pages={755\ndash 777},
      review={\MR{86m:46054}},
}

\bib{rw:morita}{book}{
      author={Raeburn, Iain},
      author={Williams, Dana~P.},
       title={Morita equivalence and continuous-trace {$C^*$}-algebras},
      series={Mathematical Surveys and Monographs},
   publisher={American Mathematical Society},
     address={Providence, RI},
        date={1998},
      volume={60},
        ISBN={0-8218-0860-5},
      review={\MR{2000c:46108}},
}

\bib{wil:crossed}{book}{
      author={Williams, Dana~P.},
       title={Crossed products of {$C{\sp \ast}$}-algebras},
      series={Mathematical Surveys and Monographs},
   publisher={American Mathematical Society},
     address={Providence, RI},
        date={2007},
      volume={134},
        ISBN={978-0-8218-4242-3; 0-8218-4242-0},
      review={\MR{MR2288954 (2007m:46003)}},
}

\bib{yam:xx87}{unpublished}{
      author={Yamagami, Shigeru},
       title={On the ideal structure of {$C^*$}-algebras over locally compact
  groupoids},
        date={1987},
        note={(Unpublished manuscript)},
}

\bib{yam:pm90}{incollection}{
      author={Yamagami, Shigeru},
       title={On primitive ideal spaces of {\cs}-algebras over certain locally
  compact groupoids},
        date={1990},
   booktitle={Mappings of operator algebras ({P}hiladelphia, {PA}, 1988)},
      series={Progr. Math.},
      volume={84},
   publisher={Birkh\"auser Boston},
     address={Boston, MA},
       pages={199\ndash 204},
}

\end{biblist}
\end{bibdiv}
\end{document}